\documentclass{article}
\usepackage[utf8]{inputenc}
\usepackage{amsmath, amsthm, amssymb, amsfonts, authblk, color, makecell}
\usepackage[T1]{fontenc}

\newtheorem{proposition}{Proposition}
\newtheorem{definition}{Definition}
\newtheorem{lemma}{Lemma}
\newtheorem{theorem}{Theorem}
\newtheorem{corollary}{Corollary}
\newtheorem{remark}{Remark}
\newtheorem{example}{Example}

\newcommand{\Ss}{\mathcal{S}}
\newcommand{\Rr}{\mathcal{R}}
\newcommand{\Tt}{\mathcal{T}}
\newcommand{\Ww}{\mathcal{W}}
\newcommand{\Dd}{\mathcal{D}}

\begin{document}

\title{On the Algebraic Properties of \(r\)-circulant Matrices Associated with Generalized \(k\)-Pell-Tribonacci Numbers}
\author{Marko Pe\v sovi\' c and Sonja Telebakovi\' c Oni\' c}
\affil{Faculty of Civil Engineering, University of Belgrade,\\ Bulevar kralja Aleksandra 73,\\ Belgrade, Serbia

\vspace{1ex}

Faculty of Mathematics, University of Belgrade,\\Studentski trg 16,\\ Belgrade, Serbia

\vspace{1ex}

\texttt{mpesovic@grf.bg.ac.rs},
\texttt{sonja.onic@matf.bg.ac.rs }}

\date{}
\maketitle

\vspace{-3ex}

\begin{abstract}
This study examines the properties of an $r$-circulant matrix whose entries are defined by the generalized $k$-Pell-Tribonacci sequence $\{P_{k,n}\}$.
Explicit expressions are derived for the Frobenius (Euclidean) norm and the entrywise $\ell_1$-norm, together with closed-form formulas for the eigenvalues and the determinant of the matrix. Furthermore, upper and lower bounds for the spectral norm are established, yielding results that generalize previously reported ones corresponding to particular sequences while also providing sharper bounds for the considered norms.

\vspace{0.5cm}
\end{abstract}

\section{Introduction}

Circulant matrices and their generalizations, particularly $r$-circulant matrices, have been extensively studied in the literature due to their rich algebraic structure and numerous applications in signal processing, coding theory, image processing, and numerical linear algebra \cite{Kome, Shen}. An $r$-circulant matrix is a natural generalization of the classical circulant matrix, see \cite{Davis,Kra}, where each row is obtained by shifting the previous row to the right by one position and multiplying the shifted element by a fixed parameter $r$. 

\medskip

Circulant matrices and their generalizations are closely related to the discrete Fourier transform (DFT). In particular, they are diagonalizable by the Fourier matrix, which leads to efficient spectral decomposition and enables fast computation in applications such as convolution and frequency-domain analysis. This highlights the importance of determining their eigenvalues, since many structural properties of $r$-circulant matrices are governed by their spectrum. Moreover, invertibility follows naturally as a consequence of this characterization, as an $r$-circulant matrix is invertible if and only if all of its eigenvalues are nonzero.

\medskip

The spectral norm and the Frobenius norm of matrices play a fundamental role in matrix analysis. The spectral norm of a matrix $A$, denoted by $\|A\|_2$, is defined as the largest singular value of $A$, while the Frobenius norm $\|A\|_F$ is the square root of the sum of the squares of all its entries. In addition, the entrywise $\ell_1$-norm, denoted by $\|A\|_{1}$, is defined as the sum of the absolute values of all entries of $A$. Establishing sharp upper and lower bounds (or even explicit formulas) for these norms, especially for structured matrices such as $r$-circulant matrices whose entries are generated by special integer sequences, has attracted considerable attention in recent years.

\medskip

In the last two decades, many researchers have investigated the norms of circulant-type matrices involving well-known number sequences and orthogonal polynomials. For instance, the spectral and Frobenius norms of circulant and $r$-circulant matrices with recurrence sequences of the second-order: Fibonacci numbers, Lucas numbers, Pell numbers, Jacobsthal numbers, and their various generalizations and combinations, including Chebyshev polynomials, have been extensively studied (see, e.g., \cite{Dagli, Kome, Merikoski, Oteles, Puc2, Radicic1, Radicic2, Radicic3,  Shen, Yazlik} and the references therein).

\medskip

Linear recurrence sequences of the third order, such as the Narayana sequence, Tribonacci sequence, third-order Pell sequence, and their various generalizations, have become a focal point of research due to their intricate algebraic structures and applications in various mathematical domains \cite{Puc, Soykan4, Soykan5}. The theoretical foundation for such sequences was significantly advanced by Shannon and Horadam \cite{Shannon}, who investigated the fundamental properties of third-order recurrence relations.

\medskip

In this paper, we focus on a sequence that combines features of both Pell and Tribonacci generalizations, which we refer to as the \emph{generalized $k$-Pell-Tribonacci sequence}. While generalized Tribonacci sequences with arbitrary coefficients \cite{Soykan1} and higher-order $k$-Pell numbers \cite{Oteles, Taher} have been studied, as well as special cases of third-order Pell numbers \cite{Soykan4}, the particular parametric family considered here is new.

\medskip

Building on this framework, we first present closed-form summation formulas for these numbers, including weighted sums, sums of squares, and weighted sums of squares, thereby enriching the analytical theory of $k$-generalized third-order sequences.

\medskip

The main aim of the paper is to investigate the Frobenius, entrywise $\ell_1$, and spectral norms of $r$-circulant matrices whose entries are drawn from the generalized $k$-Pell-Tribonacci sequence. Explicit formulas and sharp bounds for these norms are derived, and related structural properties are discussed. Furthermore, the obtained results are compared with existing formulas for third-order Pell numbers, demonstrating that the derived estimates provide significantly improved bounds. Additionally, sufficient conditions for the invertibility of an $r$-circulant matrix, for $r \in \mathbb{R}\setminus{0}$, are established.

\section{The generalized $k$-Pell-Tribonacci sequence}

\begin{definition}
For $k \in \mathbb{N}$, the generalized $k$-Pell-Tribonacci sequence $\{P_{k,n}\}_{n \in \mathbb{N}_0}$ is defined by the third-order linear recurrence relation:
\begin{equation}\label{eq:recurrence}
 P_{k,n} := 2k P_{k,n-1} + k P_{k,n-2} + P_{k,n-3}, \quad n \geq 3
\end{equation}
subject to the initial conditions $P_{k,0} = 0, P_{k,1} = 1,$  and $ P_{k,2} = 2k$.
\end{definition}

The behavior of the sequence is determined by the roots of the associated characteristic polynomial:
\begin{equation}\label{kar_jed}
\phi(x):=x^3 - 2kx^2 - kx - 1.
\end{equation}
To apply Cardano's method, we reduce the equation $\phi(x)=0$ to the depressed cubic form $y^3 + py + q = 0$.
Substituting $x = y + \frac{2k}{3}$ into the characteristic equation we get:
\begin{equation*}
\left(y + \frac{2k}{3}\right)^3 - 2k\left(y + \frac{2k}{3}\right)^2 - k\left(y + \frac{2k}{3}\right) - 1 = 0.
\end{equation*}
Expanding and simplifying, we identify the parameters $p$ and $q$:
\begin{equation*}
p = -\frac{4k^2 + 3k}{3},
\qquad
q = -\frac{16k^3 + 18k^2 + 27}{27},
\end{equation*}
so the discriminant of the associated cubic equation is given by:
\begin{equation*}
\Delta = \left(\frac{q}{2}\right)^2 + \left(\frac{p}{3}\right)^3=\dfrac{-4k^4+28k^3+36k^2+27}{108}.
\end{equation*}
Following Cardano's method, we define two intermediate variables:
\begin{equation*}
u = \sqrt[3]{-\frac{q}{2} + \sqrt{\Delta}}, \quad v = \sqrt[3]{-\frac{q}{2} - \sqrt{\Delta}}.
\end{equation*}
Since $-4k^4+28k^3+36k^2+27 \neq 0$ for all $k\in\mathbb{N}$, we have $\Delta \neq 0$ and $u \neq v$, and consequently the cubic equation $y^3 + py + q = 0$ has three distinct roots:
\begin{equation*}
y_1=u+v,\qquad y_2=\omega_3 u + \omega^2_3 v, \qquad y_3 = \omega^2_3 u + \omega_3 v,
\end{equation*}
where $\omega_3 = e^{i\frac{2\pi}{3}} = -\frac{1}{2} + i\frac{\sqrt{3}}{2}$.
Finally, substituting back for $x$, the roots of the $k$-Pell-Tribonacci characteristic polynomial, denoted by $x_1, x_2, x_3$, are distinct and given by:
\begin{equation*}
x_1 = \frac{2k}{3} + (u + v), \qquad
x_{2,3} = \frac{2k}{3} - \frac{u+v}{2} \pm i\frac{\sqrt{3}(u-v)}{2}.
\end{equation*}
In the sequel, we denote the roots $x_1, x_2, x_3$ by $\alpha, \beta, \gamma$, respectively.

\begin{proposition}\label{n_th}
The $(n+1)$-th term of the sequence $\{P_{k,n}\}_{n\in\mathbb{N}_0}$ is given by the Binet-style formula:
\begin{equation*}
 P_{k,n} = \frac{\alpha^{n+1}}{(\alpha-\beta)(\alpha-\gamma)} + \frac{\beta^{n+1}}{(\beta-\alpha)(\beta-\gamma)} + \frac{\gamma^{n+1}}{(\gamma-\alpha)(\gamma-\beta)}
\end{equation*}
where $\alpha, \beta, \gamma$ are the distinct roots of the characteristic polynomial $\phi(x)$.
\end{proposition}

\begin{proof}
According to Vieta's formulas, the roots $\alpha, \beta, \text{ and } \gamma$ of $\phi(x)$ satisfy:
\begin{equation}\label{eq:Viet}
\alpha + \beta + \gamma = 2k.
\end{equation}
The general solution of a third-order linear homogeneous recurrence relation is given by:
\begin{equation*} \label{eq:gensol}
P_{k,n} = A\alpha^n + B\beta^n + C\gamma^n,
\end{equation*}
where $A$, $B$, and $C$ are constants determined by the initial conditions. 
By substituting $n=0,1,2$ into the above equation, we obtain the following system of linear equations:
\begin{align*}
A + B + C = 0, \qquad
A\alpha + B\beta + C\gamma = 1, \qquad
A\alpha^2 + B\beta^2 + C\gamma^2 = 2k.
\end{align*}
Applying Cramer's rule together with Vieta’s relations (\ref{eq:Viet}), we obtain:
\[
A=\dfrac{ -(\gamma - \beta)(\gamma + \beta) + 2k(\gamma - \beta)}{ (\beta - \alpha)(\gamma - \alpha)(\gamma - \beta)}
= \frac{\alpha}{(\beta - \alpha)(\gamma-\alpha)}.
\]
The coefficients $B$ and $C$ can be obtained analogously by exploiting the symmetry of the roots.
\end{proof}

The growth rate of the $k$-Pell-Tribonacci sequence $\{P_{k,n}\}_{n\in\mathbb{N}_0}$ is governed by the roots of its characteristic polynomial $\phi(x)$.

\begin{lemma}
For any $k\in\mathbb{N}$, the characteristic polynomial $\phi(x)$ has exactly one positive real root $\alpha$ such that $2k < \alpha < 2k + 1$. This root is the dominant root.
\end{lemma}

\begin{proof} We observe the following values:
\begin{itemize}
\item $\phi(2k) = (2k)^3 - 2k(2k)^2 - k(2k) - 1 = -2k^2 - 1 < 0$,
\item $\phi(2k+1) = (2k+1)^3 - 2k(2k+1)^2 - k(2k+1) - 1=2k^2+3k>0$.
\end{itemize}
Since $\phi(x)$ is continuous and changes sign between $2k$ and $2k+1$, by the Intermediate Value Theorem, there exists at least one real root $\alpha \in (2k, 2k+1)$. By analyzing the derivative $\phi'(x) = 3x^2 - 4kx - k$, it can be shown that $\phi(x)$ is strictly increasing for $x > 2k$, ensuring the uniqueness and dominance of $\alpha$.
\end{proof}

In Soykan's paper \cite{Soykan1}, closed-form expressions for summation formulas of generalized Tribonacci numbers are established. As a special case, we obtain summation formulas for the $k$-Pell–Tribonacci numbers. The proof presented here uses a more direct approach.

\begin{proposition}
Let $\{P_{k,n}\}_{n \in \mathbb{N}_0}$ be the generalized $k$-Pell-Tribonacci sequence.
The sum of the first $n+1$ terms of the sequence is given by:
\begin{equation*}
\Ss^{(1)}_n :=\sum_{i=0}^{n} P_{k,i}= \frac{P_{k,n+3} + (1-2k)P_{k,n+2} + (1-3k) P_{k,n+1} - 1}{3k}.
\end{equation*}
\end{proposition}

\begin{proof}
Summing the recurrence relation $P_{k,i} = P_{k,i+3} -2k P_{k,i+2} -k P_{k,i+1}$ from $i=0$ to $n$, we get:
\begin{equation}\label{eq:sum}
\Ss^{(1)}_n=\sum_{i=0}^n P_{k,i} = \sum_{i=0}^n P_{k,i+3} - 2k \sum_{i=0}^n P_{k,i+2} - k \sum_{i=0}^n P_{k,i+1}.
\end{equation}
Expressing each summation in terms of the total sum $\Ss_n^{(1)}$, we obtain:
\begin{align*}
\sum_{i=0}^n P_{k,i+3} &= \Ss^{(1)}_n + P_{k,n+1} + P_{k,n+2} + P_{k,n+3} - (P_{k,0} + P_{k,1} + P_{k,2}),\\
\sum_{i=0}^n P_{k,i+2} &= \Ss^{(1)}_n + P_{k,n+1} + P_{k,n+2} - (P_{k,0} + P_{k,1}),\\
\sum_{i=0}^n P_{k,i+1} &= \Ss^{(1)}_n + P_{k,n+1} - P_{k,0}.
\end{align*}
Substituting the initial conditions $P_{k,0} = 0$, $P_{k,1} = 1$, and $P_{k,2} = 2k$ into the sum equation (\ref{eq:sum}) and grouping the $\Ss^{(1)}_n$ terms on the left-hand side, we have:
\begin{equation*}
3k \Ss^{(1)}_n = P_{k,n+3} + (1-2k)P_{k,n+2} + (1-3k)P_{k,n+1} - 1.
\end{equation*}
Dividing by $3k$ yields the final closed-form expression for the sum. 
\end{proof}

We aim to find the closed-form expression for the weighted sum of $k$-Pell-Tribonacci numbers, using the index shifting method.

\begin{proposition}
Let $\{P_{k,n}\}_{n \in \mathbb{N}_0}$ be the generalized $k$-Pell-Tribonacci sequence. The weighted sum $\Ww^{(1)}_n := \sum_{i=0}^{n} i P_{k,i}$ is given by:
\begin{equation*}
\Ww^{(1)}_n = \frac{(a_1n + a_2)P_{k,n+3} + (b_1 n + b_2)P_{k,n+2} + (c_1 n + c_2)P_{k,n+1} + d}{9k^2},
\end{equation*}
where $a_1=3k$, $a_2=5k-3$, $b_1=3k-6k^2$, $b_2=-10k^2+8k-3$, $c_1=3k-9k^2$, $c_2=-9k^2+8k-3$, and $d=k+3$.
\end{proposition}

\begin{proof}
We apply a weighted summation to the recurrence relation $P_{k,i} = P_{k,i+3} - 2k P_{k,i+2} - k P_{k,i+1}$ by multiplying both sides by $i$ and summing over $i=0$ to $n$. This yields:
\begin{equation*}
\Ww^{(1)}_n = \sum_{i=0}^{n} i P_{k,i} = \sum_{i=1}^{n} i \left( P_{k,i+3} - 2k P_{k,i+2} - k P_{k,i+1} \right)
=W_1-2kW_2-kW_3.
\end{equation*}
We shift the indices of the summations on the right-hand side to align them with $\Ww^{(1)}_n$ and $\Ss^{(1)}_n$, where $\Ss^{(1)}_n$ is the standard sum $\sum_{i=0}^{n} P_{k,i}$.
\begin{enumerate}
\item For the first sum $W_1$, let $j=i+3$. Then $W_1= \sum_{j=4}^{n+3} (j-3) P_{k,j}.$ We now express both sums in terms of $\Ww^{(1)}_n$ and $\Ss^{(1)}_n$:
\begin{align*}
\sum_{j=4}^{n+3} j P_{k,j}
&= \Ww^{(1)}_n + (n+1)P_{k,n+1} + (n+2)P_{k,n+2} + (n+3)P_{k,n+3}\\
&- (P_{k,1} + 2P_{k,2} + 3P_{k,3}),\\
\sum_{j=4}^{n+3} P_{k,j}
&= \Ss^{(1)}_n + P_{k,n+1} + P_{k,n+2} + P_{k,n+3}
- (P_{k,1} + P_{k,2} + P_{k,3}).
\end{align*}
\item For the second sum, we have $W_2=\sum_{j=3}^{n+2} (j-2) P_{k,j}$. Hence,
\begin{align*}
W_2&=\Ww^{(1)}_n +(n+1)P_{k,n+1}+ (n+2)P_{k,n+2}-(2P_{k,2}+P_{k,1})\\
&-2(\Ss^{(1)}_n+P_{k,n+1}+ P_{k,n+2}-(P_{k,1}+P_{k,2})).
\end{align*}
\item For the third sum, we have $W_3 = \sum_{j=2}^{n+1} (j-1) P_{k,j}$. Hence,
\begin{align*}
W_3=\Ww^{(1)}_n +(n+1)P_{k,n+1}-P_{k,1}-(\Ss^{(1)}_n+P_{k,n+1}-P_{k,1}).
\end{align*}
\end{enumerate}
Substituting the obtained expressions for $W_1$, $W_2$, and $W_3$, along with the initial conditions
$P_{k,0}=0$, $P_{k,1}=1$, and $P_{k,2}=2k$, we obtain that $3k\Ww^{(1)}_n$ equals:
\begin{equation*}
(5k-3)\Ss^{(1)}_{n} + nP_{k,n+3}+(n(1-2k)-1)P_{k,n+2}+(n(1-3k)+2k-2)P_{k,n+1}+2.
\end{equation*}
Using the expression for the partial sum $\Ss^{(1)}_n$, we obtain the required formula after division by $3k$.
\end{proof}

We will now offer a short derivation of the sum of the squares of the $k$-Pell-Tribonacci numbers. Our method avoids generating
functions and instead relies on elementary recurrence identities and telescoping sums.
The same results have been obtained by Soykan in \cite{Soykan2}.

\begin{proposition}\label{prop:sumsquares}
Let $\{P_{k,n}\}_{n \in \mathbb{N}_0}$ be the generalized $k$-Pell-Tribonacci sequence. The sum of the squares of the first $n+1$ terms is given by:
\begin{equation*}
\begin{split}
\Ss^{(2)}_n:=\sum_{i=0}^nP_{k,i}^2&= \frac{a_1P_{k,n+3}^{2}+a_2P_{k,n+2}^{2}+a_3P_{k,n+1}^{2}}{-3k(k+2)}\\
&-\frac{b_1P_{k,n+1}P_{k,n+2}+b_2P_{k,n+2}P_{k,n+3}+b_3P_{k,n+1}P_{k,n+3}+c_1}{3k(k+2)},
\end{split}
\end{equation*}
where $a_1=1$, $a_2=4k^2+4k+1$, $a_3=3k^2+6k+1$, $b_1=2k-2$, $b_2=-4k-2$, $b_3=-2$, and $c_1=-1$.
\end{proposition}

\begin{proof}
Let $\Ss^{(2)}_n:= \sum_{i=0}^{n} P_{k,i}^2,$ $\Rr_n := \sum_{i=0}^{n} P_{k,i} P_{k,i+1}$, $\Tt_n := \sum_{i=0}^{n} P_{k,i} P_{k,i+2}$.
Squaring the recurrence relation yields:
\begin{equation*}
P_{k,i+3}^2 = 4k^2 P_{k,i+2}^2 + k^2 P_{k,i+1}^2 + P_{k,i}^2 + 4k^2 P_{k,i+2} P_{k,i+1} + 4k P_{k,i+2} P_{k,i} + 2k P_{k,i+1} P_{k,i}.
\end{equation*}
Multiplying the same recurrence relation by $P_{k,i+2}$ we get:
\begin{equation*}
P_{k,i+2} P_{k,i+3} = 2k P_{k,i+2}^2 + k P_{k,i+1} P_{k,i+2} + P_{k,i} P_{k,i+2}.
\end{equation*}
Similarly, multiplying by $P_{k,i+1}$ gives:
\begin{equation*}
P_{k,i+1} P_{k,i+3} = 2k P_{k,i+1} P_{k,i+2} + k P_{k,i+1}^2 + P_{k,i} P_{k,i+1}.
\end{equation*}
Summing these identities from $i=0$ to $n$ and rearranging the indices (e.g., $\sum_{i=0}^n P_{k,i+3}^2 =\Ss^{(2)}_n - (P_{k,0}^2 + P_{k,1}^2 + P_{k,2}^2) + (P_{k,n+1}^2 + P_{k,n+2}^2 + P_{k,n+3}^2)$) leads to a system of three linear equations, represented in matrix form as:
\begin{equation*}
\mathbf{M} \begin{bmatrix} \Ss^{(2)}_n & \Rr_n & \Tt_n \end{bmatrix}^T = \begin{bmatrix} \Delta_1 & \Delta_2 & \Delta_3 \end{bmatrix}^T,
\end{equation*}
where the coefficient matrix $\mathbf{M}$ is:
\begin{equation*}
\mathbf{M} = \begin{bmatrix}
-5k^2 & - (4k^2+2k) & - 4k \\
-2k & 1-k & -1 \\
-k &- (2k+1) & 1
\end{bmatrix},
\end{equation*}
and $\Delta_1, \Delta_2, \Delta_3$ are constants calculated from the boundary terms $P_{k,0},$ $ P_{k,1},$ $ P_{k,2}$, and $P_{k,n+1}$, $P_{k,n+2},$ $ P_{k,n+3}$.
They are defined as follows:
\begin{align*}
\Delta_1 &= -P_{k,n+3}^2 - (1-4k^2)P_{k,n+2}^2 - (1-5k^2)P_{k,n+1}^2 +4k^2P_{k,n+1}P_{k,n+2}+1, \\
\Delta_2 &= 2k  P_{k,n+2}^2 +2k P_{k,n+1}^2+(k-1)P_{k,n+1}P_{k,n+2}-P_{k,n+2}P_{k,n+3},\\
\Delta_3 &= kP_{k,n+1}^2 - P_{k,n+1}P_{k,n+3} + 2kP_{k,n+1}P_{k,n+2}.
\end{align*}
Since $\det(\mathbf{M}) = -9k^2(k+2)\neq0$ for all $k\in\mathbb{N}$, the sum of squares $\Ss^{(2)}_n$ can be obtained by solving this system using Cramer's Rule:
\begin{equation*} \label{eq:Cramer}
\Ss^{(2)}_n = \frac{\det(\mathbf{M}_{\Ss^{(2)}_n})}{\det(\mathbf{M})},
\end{equation*}
where $\mathbf{M}_{\Ss_n^{(2)}}$ is the matrix $\mathbf{M}$ with the first column replaced by the vector $[\Delta_1, \Delta_2, \Delta_3]^T$.
Calculating determinant $\det(\mathbf{M}_{\Ss^{(2)}_n})$ we have:
\begin{equation*}
\begin{split}
\det(\mathbf{M}_{\Ss^{(2)}_n}) = 3k P_{k,n+3}^{2}+3k(4k^{2}+4k+1)P_{k,n+2}^{2}+3k(3k^{2}+6k+1)P_{k,n+1}^{2}+ \\
3k(2k-2) P_{k,n+1}P_{k,n+2}+3k(-4k-2)P_{k,n+2}P_{k,n+3}-6k P_{k,n+1}P_{k,n+3}-3k,
\end{split}
\end{equation*}
so dividing both the numerator and the denominator by $-3k$ yields the desired explicit formula for $\Ss^{(2)}_n$.
\end{proof}

In the next proposition we provide a detailed derivation of the closed-form expression for the weighted sum of squares, without the use of
generating functions. For yet another method we refer the reader to \cite{Soykan3}.

\begin{proposition}\label{prop:weightedsquares}
Let $\{P_{k,n}\}_{n \in \mathbb{N}_0}$ be the generalized $k$-Pell-Tribonacci sequence. 
The weighted sum of squares $\Ww_n^{(2)} := \sum_{i=0}^{n} i P_{k,i}^2$ is given by:
\begin{equation*}
\begin{split}
\Ww^{(2)}_n&= \frac{a_1P_{k,n+3}^{2}+a_2P_{k,n+2}^{2}+a_3P_{k,n+1}^{2}}{9k^2(k+2)^2}\\
&+\frac{b_1P_{k,n+1}P_{k,n+2}+b_2P_{k,n+2}P_{k,n+3}+b_3P_{k,n+1}P_{k,n+3}+c_1}{9k^2(k+2)^2},
\end{split}
\end{equation*}
where:
\begin{align*}
a_1&=-3k(k+2)n- (7k^2 + 14k + 9),\\
a_2&=-3k(k+2)(2k+1)^2n-(28k^4 + 72k^3 + 60k^2 + 20k + 9),\\ 
a_3&=-3k(k+2)(3k^2+6k+1) n - (9k^4 + 36k^3 + 46k^2 + 20k + 9),\\
b_1&=-6k(k+2)(k-1) n - (20k^3 + 20k^2 - 16k - 6),\\  
b_2&=6k(k+2)(2k+1) n + (28k^3 + 64k^2 + 46k + 6),\\ 
b_3&=6k(k+2)n + (14k^2 + 22k + 6),\\ 
c_1&=(k^{2}+2k+9).
\end{align*}
\end{proposition}

\begin{proof}
To obtain the closed form, we assume a form for $\Ww_n^{(2)}$ motivated by the structure of sums for linear recurrences:
\begin{equation*}
\Ww_n^{(2)} = \frac{W_n + \mathcal{C}}{\mathcal{D}},
\end{equation*}
where $\mathcal{C}\in\mathbb{R}$, $\mathcal{D}\in\mathbb{R}\setminus\{0\}$, and $W_n$ is the $n$-dependent part defined as:
\begin{align*}
W_n &:= \alpha_n P_{k,n+3}^2 + \beta_n P_{k,n+2}^2 + \gamma_n P_{k,n+1}^2 \\
&\quad + \eta_n P_{k,n+3}P_{k,n+2} + \theta_n P_{k,n+3}P_{k,n+1} + \mu_n P_{k,n+2}P_{k,n+1}.
\end{align*}
To determine the real coefficients $\alpha_n, \beta_n, \gamma_n, \eta_n, \theta_n$, and $\mu_n$ we use the recurrence relation $\Ww_n^{(2)} = \Ww_{n-1}^{(2)} + n P_{k,n}^2$,
which implies:
\begin{equation} \label{psi}
W_n - W_{n-1} = \mathcal{D} \, n \, P_{k,n}^2.
\end{equation}
Let $P_{k,n} = a$, $P_{k,n+1} = b$, $P_{k,n+2} = c$, and $P_{k,n+3} = 2k c + k b + a$. Substituting into the left-hand side of (\ref{psi}), expanding, and collecting coefficients of $n$ and the constant terms in the basis $\{a^2, b^2, c^2, ab, ac, bc\}$ yields the system. The coefficients of $n$ for terms other than $a^2$ must be zero, the coefficient of $n a^2$ must be $\mathcal{D}$, and all constant coefficients must be zero.
Let $\alpha_n = \alpha^{(1)} n + \alpha^{(0)}$, and similarly for the other coefficients. The resulting system is:
\begin{enumerate}
\item System for the $n$-dependent coefficients:
\begin{align*}
\alpha^{(1)} - \gamma^{(1)} &= \mathcal{D}, \\
\alpha^{(1)} k^{2} - \beta^{(1)} + \gamma^{(1)} + k \theta^{(1)} &= 0, \\
4 \alpha^{(1)} k^{2} - \alpha^{(1)} + \beta^{(1)} + 2 k \eta^{(1)} &= 0, \\
2 \alpha^{(1)} k - \mu^{(1)} + \theta^{(1)} &= 0, \\
4 \alpha^{(1)} k + \eta^{(1)} - \theta^{(1)} &= 0, \\
4 \alpha^{(1)} k^{2} + \eta^{(1)} k - \eta^{(1)} + 2 k \theta^{(1)} + \mu^{(1)} &= 0.
\end{align*}
\item System for the constant coefficients:
\begin{align*}
\alpha^{(0)} - \gamma^{(0)} + \gamma^{(1)} &= 0, \\
\alpha^{(0)} k^{2} - \beta^{(0)} + \beta^{(1)} + \gamma^{(0)} + k \theta^{(0)} &= 0, \\
4 \alpha^{(0)} k^{2} - \alpha^{(0)} + \alpha^{(1)} + \beta^{(0)} + 2 k \eta^{(0)} &= 0, \\
2 \alpha^{(0)} k - \mu^{(0)} + \mu^{(1)} + \theta^{(0)} &= 0, \\
4 \alpha^{(0)} k + \eta^{(0)} - \theta^{(0)} + \theta^{(1)} &= 0, \\
4 \alpha^{(0)} k^{2} + \eta^{(0)} k - \eta^{(0)} + \eta^{(1)} + 2 k \theta^{(0)} + \mu^{(0)} &= 0.
\end{align*}
\end{enumerate}
This is a linear system of 12 equations for the 12 coefficients, with $\mathcal{D}$ appearing in the inhomogeneous term.
The matrix $\mathbf{A}$ of the system $\mathbf{A}\mathbf{v} = \mathbf{b}$, where $\mathbf{v} = [\alpha^{(1)}, \beta^{(1)}, \gamma^{(1)}, \eta^{(1)}, \theta^{(1)}, \mu^{(1)}, \alpha^{(0)}, \beta^{(0)}, \gamma^{(0)}, \eta^{(0)}, \theta^{(0)}, \mu^{(0)}]^T$ and $\mathbf{b} = [\mathcal{D}, 0, 0, 0, 0, 0, 0, 0, 0, 0, 0, 0]^T $, is given by  a $12\times12$ block matrix $\mathbf{A}$ of the form:
\begin{equation*}
\bold{A}: = \begin{bmatrix}
\mathbf{M} & \mathbf{0} \\
\mathbf{K} & \mathbf{M}
\end{bmatrix}
\end{equation*}
where $\mathbf{M}$ and $\mathbf{K}$ are the $6\times6$ matrices:
\begin{equation*}
\mathbf{M} =\begin{bmatrix}
1 & 0 & -1 & 0 & 0 & 0 \\
k^{2} & -1 & 1 & 0 & k & 0 \\
4k^{2}-1 & 1 & 0 & 2k & 0 & 0 \\
2k & 0 & 0 & 0 & 1 & -1 \\
4k & 0 & 0 & 1 & -1 & 0 \\
4k^{2} & 0 & 0 & k-1 & 2k & 1
\end{bmatrix},\;\;
\mathbf{K} =\begin{bmatrix}
0 & 0 & 1 & 0 & 0 & 0 \\
0 & 1 & 0 & 0 & 0 & 0 \\
1 & 0 & 0 & 0 & 0 & 0 \\
0 & 0 & 0 & 0 & 0 & 1 \\
0 & 0 & 0 & 0 & 1 & 0 \\
0 & 0 & 0 & 1 & 0 & 0
\end{bmatrix}.
\end{equation*}
Using cofactor expansion we find that $\det(\mathbf{M}) = -9k^{3} - 18k^{2} = -9k^{2}(k + 2)$.
Since the upper-right block is the zero matrix, the determinant of $\mathbf{A}$ is:
$$\det(\mathbf{A}) = \det(\mathbf{M}) \cdot \det(\mathbf{M}) = 81k^{4}(k + 2)^{2}.$$
The polynomial coefficients are derived as follows:
\begin{itemize}
\item $\alpha_n = \frac{-3k^{2} - 6k}{9k^2(k+2)^2}n\Dd - \frac{7k^{2} + 14k + 9}{9k^2(k+2)^2}\Dd,$
\item $\beta_n = \frac{-12k^{4} - 36k^{3} - 27k^{2} - 6k}{9k^2(k+2)^2}n\Dd - \frac{28k^{4} + 72k^{3} + 60k^{2} + 20k + 9}{9k^2(k+2)^2}\Dd,$
\item $\gamma_n =\frac{-9k^{4} - 36k^{3} - 39k^{2} - 6k}{9k^2(k+2)^2}n\Dd - \frac{9k^{4} + 36k^{3} + 46k^{2} + 20k + 9}{9k^2(k+2)^2}\Dd,$
\item $\eta_n = \frac{12k^{3} + 30k^{2} + 12k}{9k^2(k+2)^2}n\Dd + \frac{28k^{3} + 64k^{2} + 46k + 6}{9k^2(k+2)^2}\Dd,$
\item $\theta_n = \frac{6k^{2} + 12k}{9k^2(k+2)^2}n\Dd + \frac{14k^{2} + 22k + 6}{9k^2(k+2)^2}\Dd,$
\item $\mu_n = \frac{-6k^{3} - 6k^{2} + 12k}{9k^2(k+2)^2}n\Dd - \frac{20k^{3} + 20k^{2} - 16k - 6}{9k^2(k+2)^2}\Dd.$
\end{itemize}
The constant $\mathcal{C}$ is determined by the boundary condition $\Ww_0^{(2)} = 0$, i.e.
\begin{align*}
\mathcal{C}=-W_0 &= -\alpha_0 P_{k,3}^2 - \beta_0 P_{k,2}^2 - \gamma_0 P_{k,1}^2  - \eta_0 P_{k,3}P_{k,2} - \theta_0 P_{k,3}P_{k,1} - \mu_0 P_{k,2}P_{k,1}.
\end{align*}
Substituting the initial conditions $P_{k,1}=1$, $P_{k,2}=2k$, and $P_{k,3}=4k^{2}+k$ we get $\mathcal{C}= \frac{k^{2}+2k+9}{9k^2(k+1)^2}\Dd.$
\end{proof}

\section{The $r$-circulant matrix}

\begin{definition}
Let $n\geq2,$ $r \in \mathbb{C}$ and let $\mathbf{a} = (a_0, a_1, \dots, a_{n-1})\in\mathbb{C}^{n-1}$ be a given vector. 
The $r$-circulant matrix generated by $\mathbf{a}$ is the matrix defined by:
$$\mathrm{Circ}_r(a_0, a_1, \dots, a_{n-1})  :=
\begin{bmatrix}
a_0 & a_1 & a_2 & \dots & a_{n-1} \\
r a_{n-1} & a_0 & a_1 & \dots & a_{n-2} \\
r a_{n-2} & r a_{n-1} & a_0 & \dots & a_{n-3} \\
\vdots & \vdots & \vdots & \ddots & \vdots \\
r a_1 & r a_2 & r a_3 & \dots & a_0
\end{bmatrix}.
$$
\end{definition}
Two important special cases occur when $r=1$ and $r=-1$. In the case $r=1$, the $r$-circulant matrix reduces to a classical \emph{circulant matrix}, 
whereas for $r=-1$ it becomes a \emph{skew-circulant matrix}. Additionally, when $r=0$, the $r$-circulant matrix reduces to an upper triangular matrix. Circulant matrices have been widely studied due to their rich algebraic structure and applications, with foundational contributions by Davis, see \cite{Davis}. For further details, see \cite{Kra}.

\medskip

Throughout the paper, we study $r$-circulant matrices whose generating vector is formed by the sequence $(P_{k,0}, P_{k,1}, \ldots, P_{k,n-1})$ for $n\geq2$. 
Such matrices will be denoted by: 
$$\bold{P}_n:=\mathrm{Circ}_r(P_{k,0},P_{k,1},\ldots,P_{k,n-1}).$$

\subsection{Frobenius and entrywise $\ell_1$-norms}

The Frobenius norm is a matrix norm that provides a measure of the “magnitude” of the matrix entries. 
For $r$-circulant matrices, this norm exhibits a structured symmetry that allows for a reduction from $n^2$ terms to a weighted sum of $n$ terms.
Similarly, the entrywise $\ell_1$-norm, defined as the sum of the absolute values of all entries, admits a simplified representation due to the same structural properties of $r$-circulant matrices.

\medskip

Let $n\geq2$ and $\bold{A}_n = [a_{i,j}] \in \mathbb{C}^{n \times n}$. 
The Frobenius norm and the entrywise $\ell_1$-norm of $\bold{A}_n$ are given by:
\begin{equation*} \label{norms_def}
\|\bold{A}_n\|_F = \sqrt{\sum_{i=1}^{n} \sum_{j=1}^{n} |a_{i,j}|^2}, 
\qquad
\|\bold{A}_n\|_1 = \sum_{i=1}^{n} \sum_{j=1}^{n} |a_{i,j}|.
\end{equation*}

To evaluate the double summation for $\bold{P}_n=\mathrm{Circ}_r(P_{k,0}, P_{k,1}, \ldots, P_{k,n-1})$, we categorize the entries based on their occurrence in the matrix:
\begin{enumerate}
\item 
For a fixed index $l \in \{0, 1, \dots, n-1\}$, the element $P_{k,l}$ appears in the upper triangular part (including the diagonal) exactly $n-l$ times. In these instances, the entries are not multiplied by $r$.
\item 
Due to the $r$-cyclic shift property, the same element $P_{k,l}$ appears in the strictly lower triangular part exactly $l$ times, where each entry is multiplied by the parameter $r$.
\end{enumerate}

We compute the Frobenius and entrywise $\ell_1$-norms of $\bold{P}_n$ by summing the squares and absolute values of its entries, respectively. We obtain:
\begin{align*}
\|\bold{P}_n\|_F^2
&= \sum_{j=0}^{n-1} \left( (n-j)|P_{k,j}|^2 + j|r P_{k,j}|^2 \right)
= n\sum_{j=0}^{n-1} P_{k,j}^2 + (|r|^2-1)\sum_{j=0}^{n-1} jP_{k,j}^2,\\
\|\bold{P}_n\|_1
&= \sum_{j=0}^{n-1} \left( (n-j)|P_{k,j}| + j|r P_{k,j}| \right)
= n\sum_{j=0}^{n-1} |P_{k,j}| + (|r|-1)\sum_{j=0}^{n-1} j|P_{k,j}|.
\end{align*}

Since the initial values $0$, $1$, and $2k$ are nonnegative and the recurrence relation has positive coefficients, it follows by mathematical induction that $|P_{k,n}| = P_{k,n}$ for all $n \in \mathbb{N}$. This shows that both the Frobenius norm and the entrywise $\ell_1$-norm of an $r$-circulant matrix depend only on the initial $n$ terms of the sequence and the magnitude of the parameter $r$. The following result summarizes the above computations.

\begin{theorem}\label{them_1}
Let $n\geq2$ be an integer and $\bold{P}_n=\mathrm{Circ}_r(P_{k,0}, P_{k,1}, \ldots, P_{k,n-1})$ be the $r$-circulant matrix. Then,
\begin{equation*}
\|\bold{P}_n\|_F = \sqrt{n\Ss^{(2)}_{n-1} + (|r|^2 - 1)\Ww^{(2)}_{n-1}},
\end{equation*}
and
\begin{equation*}
\|\bold{P}_n\|_1 = n \Ss^{(1)}_{n-1} + (|r| - 1)\Ww^{(1)}_{n-1}.
\end{equation*}
\end{theorem}

\subsection{Spectral norm}

The spectral norm of an $n \times n$ matrix $\bold{A}_n=[a_{i,j}]\in\mathbb{C}^{n\times n}$ is defined  as follows:
\begin{equation*} \label{spektralna_def}
\|\mathbf{A}_{n}\|_2 = \sqrt{\max_{1 \le i \le n} |\lambda_{i}(\bold{A}_n^*\bold{A}_n)|},
\end{equation*}
where $\lambda_{i}(\bold{A}_n^*\bold{A}_n)$ are the eigenvalues of the matrix $\bold{A}_n^*\bold{A}_n$ and $\bold{A}_n^*$ is the conjugate transpose of $\bold{A}_n$. The following inequality holds for any matrix:
\begin{equation} \label{eq:Zielke}
\frac{1}{\sqrt{n}}\|\bold{A}_n\|_F \le \|\bold{A}_n\|_2 \le \|\bold{A}_n\|_F.
\end{equation}
For more details see \cite{Zielke}.
The maximum column length norm $c_{1}(\bold{A}_n)$ and the maximum row length norm $r_{1}(\bold{A}_n)$ of the matrix $\bold{A}_n$
are defined as follows:
\begin{equation*}
c_{1}(\bold{A}_n) := \max_{1 \le j \le n} \left( \sum_{i=1}^{n} |a_{i,j}|^{2} \right)^{\frac{1}{2}},\qquad
r_{1}(\bold{A}_n) := \max_{1 \le i \le n} \left( \sum_{j=1}^{n} |a_{i,j}|^{2} \right)^{\frac{1}{2}}.
\end{equation*}

There is a relation between $\|\bold{A}_n\|_{2}$, $c_{1}(\bold{A}_n)$ and $r_{1}(\bold{A}_n)$ norms as follows.
Additional background on matrix norms can be found in \cite{Horn2, Horn1}.

\begin{lemma}\label{nejednakosti}
For any matrices $\bold{A} = [a_{i,j}]\in M_{n,n}(\mathbb{C})$ and $\bold{B}= [b_{i,j}]\in M_{n,n}(\mathbb{C})$, we
have:
\begin{equation*}
\|\bold{A} \circ \bold{B}\|_{2} \le r_{1}(\bold{A})c_{1}(\bold{B}),
\end{equation*}
where $\bold{A} \circ \bold{B}$ is Hadamard product, defined by $\bold{A} \circ \bold{B} := [a_{i,j} b_{i,j}]$.
\end{lemma}

The next theorem presents upper and lower bounds for the spectral norm of $r$-circulant matrix $\bold{P}_n = \text{Circ}_r(P_{k,0}, P_{k,1}, \dots, P_{k,n-1})$.

\begin{theorem}\label{main_thm}
Let $n\geq2$ be an integer and $\bold{P}_n = \mathrm{Circ}_r(P_{k,0}, P_{k,1}, \dots, P_{k,n-1})$ be a $r$-circulant matrix. Then:
\begin{equation*}
\sqrt{\Ss_{n-1}^{(2)}+\frac{|r|^2-1}{n}\Ww^{(2)}_{n-1}} \leq \|\bold{P}_n\|_{2} \leq \max\{|r|,1\}\Ss^{(1)}_{n-1}.
\end{equation*}
\end{theorem}

\begin{proof}
The lower bound for the spectral norm is derived by applying formula $(\ref{eq:Zielke})$ and Theorem \ref{them_1}.
We proceed to give an upper bound for the spectral norm of $\bold{P}_n$.
\begin{enumerate}
\item For the case $|r|\geq 1$, let $\bold{A}=\mathrm{Circ}_r(P_{k,0}^{1/2},P_{k,1}^{1/2},\ldots,P_{k,n-1}^{1/2})$ and
$\bold{B}= \mathrm{Circ}_1(P_{k,0}^{1/2},P_{k,1}^{1/2},\ldots,P_{k,n-1}^{1/2})$, so that $\bold{P}_n=\bold{A}\circ\bold{B}$. Hence, we obtain:
\begin{align*}
r_{1}(\bold{A}) &= \max_{1 \le i \le n} \left( \sum_{j=1}^{n} |a_{i,j}|^{2} \right)^{\frac{1}{2}}=\sqrt{P_{k,0}+|r|^{2}\sum_{j=1}^{n-1}P_{k,j}}=
\sqrt{|r|^{2}\Ss^{(1)}_{n-1}},\\
c_{1}(\bold{B}) &= \max_{1 \le j \le n} \left( \sum_{i=1}^{n} |b_{i,j}|^{2} \right)^{\frac{1}{2}}=\sqrt{\sum_{i=1}^{n-1}P_{k,i}}=
\sqrt{\Ss^{(1)}_{n-1}}.
\end{align*}
Applying Lemma \ref{nejednakosti} we conclude that
$\|\bold{P}_n\|_{2} \le r_{1}(\bold{A})c_{1}(\bold{B})=|r|\Ss^{(1)}_{n-1}.$
\item For the case $|r|< 1$, let  $\bold{A}=\mathrm{Circ}_r(P_{k,0}^{1/2},P_{k,1}^{1/2},\ldots,P_{k,n-1}^{1/2})$ and
$\bold{B}= \mathrm{Circ}_1(P_{k,0}^{1/2},P_{k,1}^{1/2},\ldots,P_{k,n-1}^{1/2})$, so that $\bold{P}_n=\bold{A}\circ\bold{B}$. Hence, we have:
\begin{align*}
r_{1}(\bold{A}) &= \max_{1 \le i \le n} \left( \sum_{j=1}^{n} |a_{i,j}|^{2} \right)^{\frac{1}{2}}=\sqrt{\sum_{j=0}^{n-1}P_{k,j}}=
\sqrt{\Ss^{(1)}_{n-1}},\\
c_{1}(\bold{B}) &= \max_{1 \le j \le n} \left( \sum_{i=1}^{n} |b_{i,j}|^{2} \right)^{\frac{1}{2}}=\sqrt{\sum_{i=1}^{n-1}P_{k,i}}=\sqrt{\Ss^{(1)}_{n-1}}.
\end{align*}
By Lemma \ref{nejednakosti}, we obtain that
$\|\bold{P}_n\|_{2} \le r_{1}(\bold{A})c_{1}(\bold{B})=\Ss^{(1)}_{n-1}.$
\end{enumerate}

\vspace*{-5mm}
\end{proof}

\begin{example}
\emph{For $k=1$, the sequence reduces to the third-order Pell sequence, for which $r$-circulant matrices were previously studied in \cite{Soykan4}. The norm bounds derived in the Theorem \ref{main_thm} provide improvements over those established in the aforementioned work. We now present refined spectral norm estimates for the parameter values considered at the end of that paper in the case when $|r|\geq1$.}
\begin{table}[h]
\centering
\begin{tabular}{|c|c|c|c|c|c|c|c|}
\hline
\makecell{$n$} & $r$ 
& \makecell{Old lower\\bound} 
& \makecell{New lower\\bound} 
& $\|\bold{P}_n\|_2$ 
& \makecell{New upper\\bound} 
& \makecell{Old upper\\bound} \\ \hline

5 & 1      & 14.11 & 14.11 & 21.00 & 21.00  & 199.50 \\
5 & 1.08 & 14.11 & 14.35 & 22.19 & 22.68  & 152.35 \\
5 & 1.70 & 14.11 & 16.68 & 32.72 & 35.70  & 339.15 \\
5 & 2      & 14.11 & 18.03 & 38.11 & 42.00   & 399.00 \\
5 & 4      & 14.11 & 28.79 & 74.76 & 84.00 & 798.00 \\
5 & 5      & 14.11 & 34.74 & 93.20 & 105.00 & 997.50 \\ \hline

8 & 1    & 232.68 & 232.68 & 352.00  & 352.00 &54140.50 \\
8 & 1.08 & 232.68 & 239.15 & 375.06  & 380.16 &58471.74 \\
8 & 1.70 & 232.68 & 298.02 &571.06  & 598.40 &92038.85 \\
8 & 2    & 232.68 & 330.42 &668.84  & 704.00 &108281.00 \\
8 & 4    & 232.68 & 373.88 &1326.34 & 1498.00 &216562.00 \\
8 & 5    & 232.68 & 703.18 &1655.92 & 1760.00& 270702.50 \\ \hline

\end{tabular}
\caption{Lower and upper bounds together with the spectral norm $\|\bold{P}_n\|_2$ for different values of $n$ and $r$, for $|r|\geq1$.}
\label{tab:norm_bounds}
\end{table}
\end{example}

\subsection{Eigenvalues and determinant}

A central role in the analysis of $r$-circulant matrices is played by their complete spectral characterization, comprising both eigenvalues and eigenvectors. In fact, this structure leads to a full diagonalization of the matrix, since all spectral components admit explicit closed-form expressions. For more details, see \cite[Lemma 4]{Cline}.

\begin{proposition}\label{prop:eigen-general}
Let $n\geq2$ be an integer and $r\in\mathbb{C}\setminus\{0\}$. The eigenvalues of an $r$-circulant matrix $\mathrm{Circ}_r(a_0, a_1, \dots, a_{n-1})$ are given by:
\begin{equation*}
\lambda_m = \sum_{j=0}^{n-1} a_{j} \rho_m^j,\qquad m=0,1,\ldots,n-1,
\end{equation*}
where $\rho_m = r^{1/n} \omega^m_n$ and $\omega_n = e^{2\pi i / n}$ is a primitive $n$-th root of unity satisfying $\omega^n = 1$, and hence $\rho_m^n = r$.
The corresponding eigenvectors are explicitly of the form:
\begin{equation*}
\bold{v}^{(m)} = \left(1, \rho_m, \rho_m^2, \dots, \rho_m^{n-1}\right)^T, \quad m = 0,1,\dots,n-1.
\end{equation*}
\end{proposition}

An $r$-circulant matrix $\mathrm{Circ}_r(a_0, a_1, \dots, a_{n-1})$ can be naturally associated with its \emph{generating polynomial}:
\[
\Psi_a(x) := \sum_{i=0}^{n-1} a_i x^i.
\]
In particular, the corresponding circulant structure is obtained by evaluating this polynomial in the shift matrix 
$\bold{S}_n:=\mathrm{Circ}_r(0,1,0,\dots,0),$ so that the classical circulant matrix can be written as:
$$\Psi_a(\bold{S}_n) = \mathrm{Circ}_r(a_0, a_1, \dots, a_{n-1}).$$
It turns out that the eigenvalues are also determined by this polynomial, namely $\lambda_m = \Psi_a(\rho_m)$, for $m=0,1,\ldots,n-1$.

\medskip

To evaluate $\lambda_m$, we make use of the recurrence relation \eqref{eq:recurrence} and its associated characteristic polynomial $\phi(x)=x^3-2kx^2-kx-1$. For later convenience, we also introduce the polynomials:
\begin{equation*} \label{eq:denom-poly}
\Psi(x):=\sum_{i=0}^{n-1}P_{k,i}x^i,\qquad \psi(x):=1-2kx-kx^2-x^3.
\end{equation*}
Note that the polynomial $\psi(x)$ satisfies $\psi(x)=x^3\phi(1/x)$.

\begin{lemma} \label{lem:sum-identity}
Let $n\geq3$ be an integer, $\psi(x)=1-2kx-kx^2-x^3$, and $\Psi(x) = \sum_{j=0}^{n-1} P_{k,j}x^j$. Then:
\begin{equation*} \label{eq:ds-identity}
\psi(x)\Psi(x) = x-P_{k,n}x^n-(kP_{k,n-1}+P_{k,n-2})x^{n+1}-P_{k,n-1}x^{n+2}.
\end{equation*}
\end{lemma}

\begin{proof}
Multiplying $\Psi(x)$ by $\psi(x)$ we get:
\begin{equation*}
\psi(x)\Psi(x) = \Psi(x) - 2kx\Psi(x) - kx^2\Psi(x) - x^3\Psi(x).
\end{equation*}
After expanding each shifted sum and collecting like terms, all coefficients for powers $x^j$ with $3 \leq j \leq n-1$, by the recurrence \eqref{eq:recurrence}, become:
\[
P_{k,j} - 2k P_{k,j-1} - k P_{k,j-2} - P_{k,j-3} = 0.
\]
Thus, only the boundary terms ($j=0,1,2$) and tail terms ($j=n,n+1,n+2$) survive, i.e.
$$\psi(x)\Psi(x)=\sum_{j=0}^{2} c_jx^j + \sum_{j=n}^{n+2} t_jx^j.$$
Using $P_{k,0}=0$, $P_{k,1}=1$, $P_{k,2}=2k$ we get $c_{0}=P_{k,0} = 0$, $c_{1}=P_{k,1}-2kP_{k,0}=1$, $c_{2}=P_{k,2}-2kP_{k,1}-kP_{k,0}=0$.
The tail coefficients coming from the shifted sums that extend beyond $n-1$ are $t_{n}=-2k P_{k,n-1} - k P_{k,n-2} - P_{k,n-3}=-P_{k,n}$, $t_{n+1}=-k P_{k,n-1} - P_{k,n-2}$, $t_{n+2}=-P_{k,n-1}$.
\end{proof}

Using the previous result and the relation $\bold{S}_n^n = r\bold{I}_n$, where $\bold{I}_n$ denotes the identity matrix of order $n\geq3$, it follows that:
\begin{align*}
\mathrm{Circ}_r(1,-2k,-k,-1,0,\ldots,0)\bold{P}_n
&=
\mathrm{Circ}_r\bigl(-rP_{k,n},\,1-rkP_{k,n-1}-rP_{k,n-2},\\
&\hspace{2.2cm}-rP_{k,n-1},\,0,\ldots,0\bigr).
\end{align*}

The following theorem gives an explicit expression for the eigenvalues in the case of the matrices $\bold{P}_n$.
\begin{theorem} \label{thm:eigenvalues}
Let $n\geq3$, $r\in\mathbb{C}\setminus\{0\}$, $\alpha$, $\beta$, $\gamma$ be the roots of the characteristic polynomial $\phi(x)$ defined in \eqref{kar_jed}. Set $\rho_m = r^{1/n} \omega_n^m$, for $m = 0,1,\dots,n-1$ and $\omega_n = e^{2\pi i / n}$.
The eigenvalues of the matrix $\bold{P}_n = \mathrm{Circ}_r(P_{k,0}, P_{k,1}, \dots, P_{k,n-1})$ are given as follows:
\begin{enumerate}
\item If $\rho_m \notin \{\alpha^{-1}, \beta^{-1}, \gamma^{-1}\}$, then:
\[
\lambda_m =
\frac{
\rho_m - rP_{k,n}
- r\rho_m (k P_{k,n-1} + P_{k,n-2})
- r\rho_m^2 P_{k,n-1}
}{
1-2k\rho_m-k\rho_m^2-\rho_m^3
}.
\]

\item If $\rho_m = \alpha^{-1}$, then:
\[
\lambda_m =
\frac{n\alpha}{(\alpha-\beta)(\alpha-\gamma)}
-\frac{\beta(\alpha^n-\beta^n)}{\alpha^{n-1}(\beta-\alpha)^2(\beta-\gamma)}
-\frac{\gamma(\alpha^n-\gamma^n)}{\alpha^{n-1}(\gamma-\alpha)^2(\gamma-\beta)}.
\]

\item If $\rho_m = \beta^{-1}$, then:
\[
\lambda_m =
\frac{n\beta}{(\beta-\alpha)(\beta-\gamma)}
-\frac{\alpha(\beta^n-\alpha^n)}{\beta^{n-1}(\alpha-\beta)^2(\alpha-\gamma)}
-\frac{\gamma(\beta^n-\gamma^n)}{\beta^{n-1}(\gamma-\beta)^2(\gamma-\alpha)}.
\]

\item If $\rho_m = \gamma^{-1}$, then:
\[
\lambda_m =
\frac{n\gamma}{(\gamma-\alpha)(\gamma-\beta)}
-\frac{\alpha(\gamma^n-\alpha^n)}{\gamma^{n-1}(\alpha-\gamma)^2(\alpha-\beta)}
-\frac{\beta(\gamma^n-\beta^n)}{\gamma^{n-1}(\beta-\gamma)^2(\beta-\alpha)}.
\]
\end{enumerate}
\end{theorem}

\begin{proof}
From Proposition \ref{prop:eigen-general}, we have $\lambda_m = \Psi(\rho_m) = \sum_{i=0}^{n-1} P_{k,i} \rho_m^i.$
Applying Lemma \ref{lem:sum-identity} and using the identities $\rho_m^n = r$, $\rho_m^{n+1} = r \rho_m$, and $\rho_m^{n+2} = r \rho_m^2$, we obtain:
\[
\psi(\rho_m)\lambda_m = \rho_m - rP_{k,n} - r \rho_m (k P_{k,n-1} + P_{k,n-2}) - r \rho_m^2 P_{k,n-1}.
\]
Dividing by $\psi(\rho_m)$ yields the desired expression for $\lambda_m$ in the case $\psi(\rho_m)\neq 0$, i.e., when $\rho_m$ is not a root of the reciprocal characteristic polynomial.

\medskip

\noindent
We now consider the case $\rho_m=\alpha^{-1}$. By Proposition \ref{n_th}, we have:
\begin{align*}
\lambda_m
&= \sum_{i=0}^{n-1} P_{k,i} \rho_m^i
= \sum_{i=0}^{n-1} \frac{A\alpha^i + B\beta^i + C\gamma^i}{\alpha^i}= A \sum_{i=0}^{n-1} 1
+ B \sum_{i=0}^{n-1} \frac{\beta^i}{\alpha^i}
+ C \sum_{i=0}^{n-1} \frac{\gamma^i}{\alpha^i},
\end{align*}
where $A=\frac{\alpha}{(\alpha-\beta)(\alpha-\gamma)},$ $B=\frac{\beta}{(\beta-\alpha)(\beta-\gamma)},$ $C=\frac{\gamma}{(\gamma-\alpha)(\gamma-\beta)}.$
Using the formula for a geometric series, we obtain:
\[
\lambda_m
= An
+ B \frac{1-(\beta/\alpha)^n}{1-\beta/\alpha}
+ C \frac{1-(\gamma/\alpha)^n}{1-\gamma/\alpha}.
\]
After simplification, this yields the desired closed form.
The cases $\rho_m=\beta^{-1}$ and $\rho_m=\gamma^{-1}$ are derived analogously.
\end{proof}

To derive the closed form of the determinant of the matrix $\bold{A}_n$, we use the fact that the $r$-circulant matrix $\bold{A}_n$ is diagonalizable over $\mathbb{C}$ and that the determinant of a diagonalizable matrix equals the product of its eigenvalues.
In the generic case, when $\rho_m \notin \{\alpha^{-1}, \beta^{-1}, \gamma^{-1}\}$ for all $m = 0,1,\dots,n-1$, the eigenvalue expression admits a particularly convenient form, which leads to a more compact representation of the determinant.

\begin{theorem} \label{thm:determinant}
Let $n\geq3$, $r\in\mathbb{C}\setminus\{0\}$, and $\alpha$, $\beta$, $\gamma$ be the roots of the characteristic polynomial $\phi(x)$ defined in \eqref{kar_jed}, and let $\rho_m = r^{1/n} \omega_n^m$, where $m = 0,1,\dots,n-1$ and $\omega_n = e^{2\pi i / n}$. 
Assume that $\rho_m \notin \{\alpha^{-1}, \beta^{-1}, \gamma^{-1}\}$ for all $m = 0,1,\dots,n-1$. Then the determinant of $\bold{P}_n =  \mathrm{Circ}_r(P_{k,0}, P_{k,1}, \dots, P_{k,n-1})$ is given by:
\begin{align*} 
\det (\bold{P}_n) =(-1)^n\, \dfrac{r^nP_{k,n-1}^n(r_1^n-r)(r_2^n-r)}{(\alpha^{-n}-r)(\beta^{-n}-r)(\gamma^{-n}-r)},
\end{align*}
where $r_1$ and $r_2$ satisfy:
$$r_1+r_2=-\dfrac{rkP_{k,n-1}+rP_{k,n-2}-1}{rP_{k,n-1}},\qquad
r_1r_2=\dfrac{P_{k,n}}{P_{k,n-1}}.$$
\end{theorem}

\begin{proof}
From the spectral decomposition of the $r$-circulant matrix $\bold{P}_n$, we have  $\det(\bold{P}_n)=\prod_{m=0}^{n-1}\lambda_m$. Hence,
\[
\det(\bold{P}_n)
=
\frac{\prod_{m=0}^{n-1} \big(\rho_m - rP_{k,n}
- r \rho_m (k P_{k,n-1} + P_{k,n-2})
- r \rho_m^2 P_{k,n-1}\big)}
{\prod_{m=0}^{n-1} \big(1 - 2k \rho_m - k \rho_m^2 - \rho_m^3\big)}.
\]
The numerator can be viewed as a quadratic polynomial in $\rho_m$ that takes the form:
\[
-rP_{k,n-1}(\rho_m - r_1)(\rho_m - r_2),
\]
where $r_1,r_2$ are the roots determined via Vieta’s relations:
$$r_1+r_2=-\dfrac{rkP_{k,n-1}+rP_{k,n-2}-1}{rP_{k,n-1}},\qquad
r_1r_2=\dfrac{P_{k,n}}{P_{k,n-1}}.$$
Using the identity:
\begin{equation}\label{new_id}
\prod_{m=0}^{n-1} (\rho_m - a) = (-1)^n (a^n - r),
\end{equation}
we obtain that numerator equlas $(-1)^nr^nP_{k,n-1}^n (r_1^n - r)(r_2^n - r).$

\medskip

\noindent
The denominator polynomial factorizes as:
\[
1 - 2k\rho_m - k\rho_m^2 - \rho_m^3
= -(\rho_m - r_3)(\rho_m - r_4)(\rho_m - r_5),
\]
where $r_3,r_4,r_5$ are roots of  the polynomial $\psi(x)=1-2kx-kx^2-x^3$, i.e. $r_3=\alpha^{-1}$, $r_4=\beta^{-1}$, and $r_5=\gamma^{-1}.$ Using $(\ref{new_id})$ we obtain that the values for the denominator equals $(\alpha^{-n}-r)(\beta^{-n}-r)(\gamma^{-n}-r).$
This completes the proof.
\end{proof}

A characterization of the invertibility of $r$-circulant matrices is given in \cite{Puc}. In particular, the following proposition provides a necessary and sufficient condition in terms of the associated polynomial.

\begin{proposition} Let $n\geq2$. The $r$-circulant matrix $\mathrm{Circ}_r(a_0,a_1,a_2,\dots,a_{n-1})$ for $r\in\mathbb{C}\setminus\{0\}$ is invertible if and only if:
$$\mathrm{gcd}(\Psi_{a}(x), x^n - r) = 1,$$
where $\Psi_a(x)$ is the corresponding generating polynomial.
\end{proposition}

In the sequel, we use this proposition to study invertibility. In particular, we apply it to analyze the invertibility of both circulant and skew-circulant matrices associated with the generalized $k$-Pell–Tribonacci sequences.

\begin{theorem} Let $n\geq 2$ be an integer and $\alpha$ be the real root of $\phi(x)$. For all $r\in\mathbb{R}^+\setminus\{\alpha^{-1},P_{k,n-1}^{-n/2}P_{k,n}^{n/2}\}$, the $r$-circulant matrix: 
$$\bold{P}_n=\mathrm{Circ}_r(P_{k,0},P_{k,1},\ldots,P_{k,n-1})$$ is invertibile.
\end{theorem}

\begin{proof} 
Note that $\mathrm{det}(\mathrm{Circ}_r(0,1))=-r<0$, so let us suppose that $n\geq3$. Since $r\in\mathbb{R}^+$, we can take $\rho_m=r^{1/n}\omega_n^m$, for $m<n$. The corresponding eigenvalue $\lambda_m$, for $m=0,1,…,n-1$, is given by:
\[
\frac{
r^{1/n}\omega_n^{m} - rP_{k,n}
- rr^{1/n}\omega_n^{m} (k P_{k,n-1} + P_{k,n-2})
- rr^{2/n}\omega_n^{2m} P_{k,n-1}
}{
1-2kr^{1/n}\omega_n^{m}-kr^{2/n}\omega_n^{2m}-r^{3/n}\omega_n^{3m}}.
\]
Since $\mathrm{det}(\bold{P}_n)=\prod_{m=0}^{n-1}\lambda_m$, the matrix $\bold{P}_n$ is singular if and only if:
\begin{align*}\label{r_plus}
 rP_{k,n}+r^{1/n}\omega_n^{m}(rkP_{k,n-1}+rP_{k,n-2}-1) 
+rr^{2/n}P_{k,n-1}(\omega_n^{m})^2=0
\end{align*}
for some $m=0,1,\ldots n-1.$ Assume that there exists $m$ such that the previous equation holds.  Let us consider the following cases:
\begin{enumerate}
\item
If $m=0$, then $ rP_{k,n}+r^{1/n}(rkP_{k,n-1}+rP_{k,n-2}-1)+ rr^{2/n}P_{k,n-1}$  must be zero. This is a contradiction, since all terms are positive. 
\item
If $m\neq0$, observing that $|\omega_n^m| = 1$, it follows that $P_{k,n} = r^{2/n} P_{k,n-1}.$
This contradicts the hypothesis of the theorem.
\end{enumerate}

\vspace*{-6mm}
\end{proof}

\begin{corollary}
For all $n\geq 2$, the circulant matrix: 
$$\mathrm{Circ}_1(P_{k,0},P_{k,1},\ldots,P_{k,n-1})$$ 
is invertibile.
\end{corollary}

\begin{remark}
\emph{A numerical experiment implemented in Python was conducted for $n = 2, 3, \ldots, 30$ and $k = 1, 2, \ldots, 10$, in the cases not covered by the previous theorem. While the results indicate invertibility in most cases, counterexamples were identified (e.g., $k=5$ for $n=28,29,30$), showing that the statement does not hold in full generality.}
\end{remark}

\begin{theorem} Let $n\geq 2$ be an integer. For all $r\in\mathbb{R}^-\setminus\{-P_{k,n-1}^{-n/2}P_{k,n}^{n/2}\}$, 
the $r$-circulant matrix: 
$$\bold{P}_n=\mathrm{Circ}_r(P_{k,0},P_{k,1},\ldots,P_{k,n-1})$$ 
is invertibile.
\end{theorem}

\begin{proof} 
Note that $\mathrm{det}(\mathrm{Circ}_r(0,1))=-r>0$, so let us suppose that $n\geq3$. Since $r\in\mathbb{R}^-$, we can take $\rho_m=|r|^{1/n}\omega_n^{i/2}\omega_n^m$, for some positive integer $i,m$ such that $i/2+m\leq n$. The eigenvalue $\lambda_m$, for $m=0,1,…,n-1$, is given by:
{\small\[
\frac{
|r|^{1/n}\omega_n^{i/2+m} +|r|P_{k,n}
+|r|^{1+1/n}\omega_n^{i/2+m} (k P_{k,n-1} + P_{k,n-2})
+|r|^{1+2/n}\omega_n^{i+2m} P_{k,n-1}
}{
1-2k|r|^{1/n}\omega_n^{i/2+m}-k|r|^{2/n}\omega_n^{i+2m}-|r|^{3/n}\omega_n^{3i/2+3m}}.
\]}

\noindent
Since $\mathrm{det}(\bold{P}_n)=\prod_{m=0}^{n-1}\lambda_m$, the matrix $\bold{P}_n$ is singular if and only if:
\begin{align*}\label{r_plus}
 |r|P_{k,n}+|r|^{1/n}\omega_n^{i/2+m}(1+|r|kP_{k,n-1}+|r|P_{k,n-2}) 
+|r|^{1+2/n}P_{k,n-1}(\omega_n^{i/2+m})^2=0,
\end{align*}
for some $m=0,1,\ldots n-1.$ Assume that there exists $m$ such that the previous equation holds.  Let us consider the following cases:
\begin{enumerate}
\item
If $i/2+m=n$, then $|r|P_{k,n}+|r|^{1/n}(1+|r|kP_{k,n-1}+|r|P_{k,n-2}) 
+|r|^{1+2/n}P_{k,n-1}$ must be zero. This is a contradiction, since all terms are positive. 
\item 
If $i/2+m<n$, observing that $|\omega^{i/2+m}_n| = 1$, it follows that $P_{k,n} = |r|^{2/n} P_{k,n-1}.$
This contradicts the hypothesis of the theorem.
\end{enumerate}

\vspace*{-6mm}
\end{proof}

\begin{corollary}
For all $n\geq 2$, the skew-circulant matrix: 
$$\mathrm{Circ}_{-1}(P_{k,0},P_{k,1},\ldots,P_{k,n-1})$$ 
is invertibile.
\end{corollary}

\begin{remark}
\emph{A numerical experiment implemented in Python was conducted for $n = 2, 3, \ldots, 30$ and $k = 1, 2, \ldots, 10$, in the cases not covered by the previous theorem. While the results indicate invertibility in most cases, counterexamples were identified (e.g., $k=5$ for $n=28,29,30$), showing that the statement does not hold in full generality.}
\end{remark}

\section{Acknowledgments}

Authors are deeply indebted to Professor Zoran Pucanovi\' c for having kindly suggested the problem.
This study is partially supported by the Ministry of  Science, Technological Development and Innovation, Republic of Serbia, through the project 451-03-33/2026-03/200104.

\bibliographystyle{plain}

\end{document}